\documentclass{amsart}
\usepackage{amsmath}
\usepackage{amsfonts}
\usepackage{amssymb}
\usepackage{graphicx}%

\newtheorem{myrem}{Remark}[section]
\newtheorem{mytheo}{Theorem}[section]
\newtheorem{mylem}{Lemma}[section]

\usepackage{color}
\usepackage{fancyhdr}

\theoremstyle{remark}

\numberwithin{equation}{section}

 \allowdisplaybreaks[4]
\begin{document}

\title{Quantitative estimates of the spectral norm of random matrices with independent columns}

\pagestyle{fancy}

\fancyhf{} 


\fancyhead[CO]{\footnotesize ON RANDOM  MATRICES WITH INDEPENDENT COLUMNS }

\fancyhead[CE]{\footnotesize G. DAI, Z. SU AND H. WANG }

\fancyhead[LE]{\thepage}
\fancyhead[RO]{\thepage}
\renewcommand{\headrulewidth}{0mm}

\author{Guozheng Dai}
\address{School of Mathematical Sciences, Zhejiang University, Hangzhou, 310058,  China.}
\email{guozhengdai1997@gmail.com}

\author{Zhonggen Su}
\address{School of Mathematical Sciences, Zhejiang University, Hangzhou, 310058,  China.}
\email{suzhonggen@zju.edu.cn}

\author{Hanchao Wang }
\address{School of Mathematics, Shandong University, Jinan, 250100, China.}
\email{hcwang06@gmail.com}

\date{}


\begin{abstract}
		This paper investigates the nonasymptotic properties of the spectral norm of some random matrices with independent columns.  In particular, we consider an $m\times n$ random matrix $BA$, where $A$ is an $N\times n$ random matrix with independent mean-zero subexponential entries, and $B$ is an $m\times N$ deterministic matrix. We prove that the $L_{p}$ norm of the spectral norm of $BA$ is upper bounded by $(\sqrt{m}+\sqrt{n})p$. It is remarkable that this result is independent of the dimension $N$. 
\end{abstract}

\maketitle

\section{Introduction}
The study of random matrices can be traced back to the 1920s. Wishart \cite{Wishart_Biometrika} initiated the systematic study of large random matrices  due to the need for his work on the statistical analysis of large samples. With the works of Wigner \cite{Wigner_annals_1,Wigner_annals_2} on his famous semicircle law, which significantly contributed to understanding    the spectral statistics of random matrices from an asymptotic point of view, random matrices attracted more and more attention. Nowadays, random matrix theory has become an essential branch of probability, which also has come to play an important role in many other areas \cite{introduction,Tao.matrices,tropp}.

In this paper, we concentrate on the nonasymptotic theory of random matrices. This theory is of particular importance in the applications to high-dimensional problems, such as the convex geometry \cite{Friedlad_IMRN}, the compressed sensing \cite{Rauhut}, the statistical learning theory \cite{highdimension}. The quantity we are concerned with here is the spectral norm of random matrices. Recall that the spectral norm  $\Vert A\Vert$ is defined as the largest singular value of a matrix $A$. In particular, $\Vert A\Vert=\sup_{x}\Vert Ax\Vert_{2}/\Vert x\Vert_{2}$ where $\Vert \cdot\Vert_{2}$ denotes the Euclidean norm. We are interested in the  quantitive estimates of the spectral norms of some random matrices.

\subsection{Matrices with independent entries}

Consider an $N\times n$ random matrix $A=(a_{ij})$. When $a_{ij}$ are independent random variables, the spectral norm $\Vert A\Vert$ is well studied. We next review some known results of $\Vert A\Vert$ in this setting.

Assume that $a_{ij}$ are standard gaussian variables, i.e. $A$ is a Ginibre matrix, one has (see e.g. \cite{Szarek_J_Com})
\begin{align}
	\textsf{P}\{\Vert A\Vert> C(\sqrt{N}+\sqrt{n})  \}\le \exp(-c(N+n)),\nonumber
\end{align}
where $C, c>0$ are universal constants. 

When $a_{ij}$ are  centered subgaussian random variables (not necessarily identically distributed), an $\varepsilon$-net argument yields that (see e.g. \cite{highdimension})
\begin{align} 
	\textsf{P}\big\{\Vert A\Vert> C(\sqrt{n}+\sqrt{N})\max_{i,j}\Vert a_{ij}\Vert_{\psi_{2}}\big\}\le \exp\big(-c(N+n)\big).\nonumber
\end{align} 
Here, we call a random variable $\xi$ subgaussian if there exists $K>0$ satisfying 
\begin{align}
	\textsf{E}\exp(\frac{ \xi^{2}}{K^{2}})\le 2,\nonumber
\end{align}
and the subgaussian norm is defined as follows:
\begin{align}
	\|\xi\|_{\psi_{2}}=\inf\{K>0: \textsf{E}\exp(\frac{\xi^{2}}{K^{2}})\le 2 \}.\nonumber
\end{align}

We remark that the above so-called $\varepsilon$-net argument relies heavily on the tails of the subgaussian variables, which decay square exponentally. Hence such argument fails when facing  heavy-tailed cases, such as the subexponential case. 

Recall that a random variable $\xi$ is called subexponential if there exists $K>0$ such that
\begin{align}
	\textsf{E}\exp(\frac{\vert \xi\vert}{K})\le 2,\nonumber
\end{align}
and the subexponential norm is defined by
\begin{align}
	\|\xi\|_{\psi_{1}}=\inf\{K>0: \textsf{E}\exp(\frac{| \xi|}{K})\le 2 \}.\nonumber
\end{align}
The class of subexponential distributions is important in many fields and covers a lot of commonly used distributions. For example, as follows from the Brunn-Minkowski inequality, the uniform distribution on every convex body is subexponential (see e.g. \cite{Giannopoulos_book}). 

When $a_{ij}$ are centered subexponential random variables, a direct consequence of the main result in
 \cite{dsw} (see Theorem 1.1 there) yields that
\begin{align}\label{1.1}
	\textsf{P}\{\Vert A\Vert> C(\sqrt{n}+\sqrt{N})\max_{i,j}\Vert a_{ij}\Vert_{\psi_{1}} \}\le \exp(-c(\sqrt{n}+\sqrt{N})).
\end{align}

Seginer \cite{Seginer_CPC} and Latała \cite{Latala_proceeding}  explored the nonasymptotic properties of $\Vert A\Vert$  in a more general setting. In particular, Seginer only assumed $a_{ij}$ are i.i.d. mean zero random variables and Latała
 removed the condition of identical distribution. In these settings, it is impossible to bound the $\Vert A\Vert$ with overwhelming probability as above. Hence, they both only gave  upper bounds for $\textsf{E}\Vert A\Vert$. We refer interested readers to \cite{Seginer_CPC} for these results and do not cover them here.

\subsection{Matrices with independent columns}

We next consider the random matrices with independent columns but the entries are not necessarily independent. In particular, let $W=(W_{1}, \cdots, W_{n})$ be an $m\times n$ random matrix with independent volumn vectors $W_{i}, 1\le i\le n$. We first review some known results of $\Vert W\Vert$.

When  $W_{i}$ are further assumed to be distributed identically according to an isotropic, log-concave probability measurable on $\mathbb{R}^{m}$, Adamczak et al. \cite{Adamczak_1} showed that
\begin{align}
	\textsf{P}\{\Vert W\Vert>C(\sqrt{m}+\sqrt{n})  \}\le \exp(-c(\sqrt{m}+\sqrt{n})).\nonumber
\end{align}
Here a random vector $X\in\mathbb{R}^{m}$ is called istropic if for all $y\in\mathbb{R}^{m}$
\begin{align}
	\textsf{E}X^{\top}y=0,\quad \textsf{E}(X^\top y)^{2}=\Vert y\Vert_{2}^{2},\nonumber
\end{align}
and a measure $\mu$ on $\mathbb{R}^{m}$ is log-concave if for any measure subsets $A, B$ of $\mathbb{R}^{m}$ and any $0\le \theta\le 1$,
\begin{align}
	\mu\{\theta A+(1-\theta)B  \}\ge \mu\{A\}^{\theta}\mu\{B\}^{(1-\theta)}.\nonumber
\end{align}

Another interesting case of such model can written as $W=BA$, where $A$ is an $N\times n$ random matrix with independent columns and $B$ is an $m\times N$ non-random matrix. Obviously, $W_{i}=BA_{i}$ are independent.

 When $B$ is an $N\times N$  nonnegative definite Hermitian matrix, this model has extensive  research in high-dimensional statistics. For example, suppose that $X$ follows an $N$-dimensional gaussian distribution $N(0_{N}, \Sigma_{N})$ and we want to test
$\quad \Sigma_{N}=M$. Let $Y=M^{-1/2}X$ and then we can reduce the above hypothesis test problem of $X$ to the null hypothesis test problem of $Y$. To settle this problem, we need to explore the properties of the sample covariance matrix of $Y$, i.e. 
\begin{align}
	\frac{1}{n}(Y_{1}, \cdots, Y_{n})(Y_{1}, \cdots, Y_{n})^{\top}=\frac{1}{n}M^{-1/2}(X_{1}, \cdots, X_{n})(X_{1}, \cdots, X_{n})^\top M^{-1/2},\nonumber
\end{align}
where $X_{1},\cdots,X_{n}\stackrel{i.i.d.}{\sim}X$.  Here $B=M^{-1/2}$ and $A=(X_{1}, \cdots, X_{n})/\sqrt{n}$.
 Since most results in this regard are in the limit regime, we do not cover them here and refer interested readers to \cite{baiap1,baiap2} for more information.

When $\Vert B\Vert\le 1$ and the entries of $A$ are independent, one can estimate the  smallest singular value of random rectangular matrices based on the upper bound of $\Vert BA\Vert$ (see \cite{Rudelson,ptrfvershynin} for details). For example, Rudelson and Vershynin \cite{Rudelson} used an $\varepsilon$-net argument to prove that, when $a_{ij}$ are  mean-zero subgaussian variables,
\begin{align}
	\textsf{P}\big\{ \Vert W\Vert>C(\sqrt{n}+\sqrt{N})\max_{i,j}\Vert a_{ij}\Vert_{\psi_{2}}  \big\}\le \exp\big(-c(N+n)   \big).\nonumber
\end{align}
They then applied this result to obtain a deviation inequality for the smallest singular value of random rectangular matrices with i.i.d. standard subgaussian entries.  Following \cite{Rudelson}, Vershynin \cite{ptrfvershynin} 
consider the case where  $a_{ij}$ are mean-zero variables with finite $(4+\varepsilon)$ moment. He
showed that 
\begin{align}
	\textsf{E}\Vert W\Vert\le C(\varepsilon)(\sqrt{n}+\sqrt{m}),
\end{align}
  where $C(\varepsilon)$ is a constant depending only on the parameter $\varepsilon$. With this inequality in hand, he explored the properties of the  smallest singular value of random rectangular matrices whose entries are i.i.d. standard variables with finite fourth moment.

This paper aims to study the nonasymptotic properties of $W=BA$, where $B$ is an $m\times N$ non-random matrix, and $A=(a_{ij})$ is an $N\times n$ random matrix with independent mean-zero subexponential  entries. 
Our main result concerns the $p$-th moment bound of $\Vert W\Vert$ and shows that $(\textsf{E}\Vert W\Vert^{p})^{1/p}$ is upper bounded by $(\sqrt{n}+\sqrt{m})p$. The detailed contents are as follows:

\begin{mytheo}\label{Theo_main_weak}
	Let $W=BA$ be an $m\times n$ random matrix, where $A=(a_{ij})$ is an $N\times n$ random matrix whose entries are independent subexponential random variables with mean zero, and $B$ is an $m\times N$ non-random matrix such that $\Vert B\Vert\le 1$. Then for $p\ge 1$
	\begin{align}\label{1.3}
		(\textsf{E}\Vert W\Vert^{p})^{1/p}\le C\max_{i, j}\Vert a_{ij}\Vert_{\psi_{1}}(\sqrt{m}+\sqrt{n})p,
	\end{align}
	where $C$ is a universal constant.
\end{mytheo}

The above result states that, if the entries of $A$ are independent subexponential entries, then the spectral norm of $W$ is also a subexponential random variable, see Lemma \ref{Lem_2.3} below for details.

More recently, Adamczak, Prochno, Strzelecka and Strzelecki \cite{Adamczak_2} explored the $p$-th moment bound of $\Vert W\Vert$ when $B$ is an $N\times N$ identity matrix. In particular, they showed that (see Proposition 1.16 there)
\begin{align}
	(\textsf{E}\Vert W\Vert^{p})^{1/p}\lesssim \max_{j\le n}\sqrt{\sum_{i\le N}\Vert a_{ij}\Vert_{\psi_{1}}^{2}}p,\nonumber
\end{align}
where $W$ is in the setting of Theorem \ref{Theo_main_weak}. Compared with this result,  our conclusion captures the correct bound order , $\sqrt{m}+\sqrt{n}$ (independent of $N$), with a genereal $B$.

To conclude the Introduction, let us simplify the matrix model and outline the ideas of the proof. First,  we may and do assume that $m=n< N$. Indeed,  if $n<m$ (resp. $m<n$), then we can  extend $A$ (resp. $B$) by adding  $m-n$ zero columns to $A$  (resp. $n-m$ rows to $B$) and note
	 	\begin{align}
		\sqrt{n}+\sqrt{m}\le 2\max\{\sqrt{n}, \sqrt{m}\}\le 2(\sqrt{n}+\sqrt{m}).\nonumber
	\end{align}
 If $N\le n$, then (\ref{1.1}) together with the fact $\Vert BA\Vert\le \Vert B\Vert\Vert A\Vert$ directly yields that
	\begin{align}\label{1.4}
		&\textsf{P}\big\{ \Vert W\Vert> C\max_{i,j}\Vert a_{ij}\Vert_{\psi_{1}}\sqrt{n}t\big\}\nonumber\\ \le&	\textsf{P}\big\{ \Vert A\Vert> \frac{C}{2}\max_{i,j}\Vert a_{ij}\Vert_{\psi_{1}}(\sqrt{n}+\sqrt{N})t\big\}\le e^{-c(\sqrt{n}+\sqrt{N})t},
	\end{align}
 as desired. Hence the main challenge of this paper is how to obtain  nonasymptotic results independent of $N$ when $N>n$.

  Without loss of generality, we can assume each $a_{ij}$ is a symmetric random variable. Indeed,
  let $\tilde{A}=(\tilde{a}_{ij})$ be an independent copy of $A$ and  $\{ \varepsilon_{ij}\}$ be a sequence of  i.i.d. Rademacher random variables independent of $A$ and $\tilde{A}$. Then
 	\begin{align}
 		\textsf{E}\Vert BA\Vert^{p}&=\textsf{E}\big\Vert B(A-\textsf{E}\tilde{A})\big\Vert^{p}=\textsf{E}\big\Vert \textsf{E}\big(B(A-\tilde{A})\big|A\big)\big\Vert^{p}\le \textsf{E}\big\Vert B(A-\tilde{A})\big\Vert^{p}\nonumber\\
 		&=\textsf{E}\big\Vert B\big(\varepsilon_{ij}(a_{ij}-\tilde{a}_{ij})\big)_{N\times n}\big\Vert^{p}\le 2^{p}\textsf{E}\big\Vert B(\varepsilon_{ij}a_{ij})_{N\times n}\big\Vert^{p}.\nonumber
 	\end{align}
 Note that $\varepsilon_{ij}a_{ij}$ is a symmetric subexponential variable. Hence, we can only consider the symmetric case.


 


For the simplified matrix model with $m=n< N$, we next outline the proof of  Theorem \ref{Theo_main_weak}.
 Let $I$ be an index set containing all indexes of large columns (i.e., the column norm is at least $Cn^{-1}\log^{-O(1)}n$)  of $B$,  and let $B_{I}$  be the submatrix of $B$ whose columns are in $I$ and $A_{I}$   the submatrix of $A$ whose rows are in $I$.  By adapting the ideas from Vershynin \cite{ptrfvershynin}, we have the following matrix decomposition
 \begin{align}
	W=BA=B_{I}A_{I}+B_{I^{c}}A_{I^{c}}.  \label{BII}
\end{align}
Since $\Vert B\Vert\le 1$, then   most columns, all but $Cn^{3}\log^{O(1)} n$ columns, are always at most $Cn^{-1}\log^{-O(1)}n$. We call such columns  small. So $B_I$ has at most $Cn^{3}\log^{O(1)} n$ columns left after removing  small columns so that $B_I$ looks like an almost square matrix.

Obviously, by (\ref{BII}) we have for $p\ge 1$
 \begin{align}
(\textsf{E}\Vert W\Vert^{p})^{1/p}\le (\textsf{E}\Vert B_{I}A_{I}\Vert^{p})^{1/p}+(\textsf{E}\Vert B_{I^{c}}A_{I^{c}}\Vert^{p})^{1/p}. \label{BII-2}
\end{align}
Hence we are left to estimate each term in the RHS of (\ref{BII-2}) separately.

As for the almost square matrix $B_{I}A_{I}$,  we shall apply Lemma \ref{lm32}  to estimate  $\|B_{I}A_{I}\|$. The more muscular  subexponential condition contributes to a desired moment bound  of $\Vert B_{I}A_{I}\Vert$ by a conditional argument.  This is presented in Section 3.2 in detail.

Turn to estimating $\|B_{I^{c}}A_{I^{c}}\|$. For the sake of writing, we omit the subscript $I^c$ below.  As discussed above, the $\varepsilon$-net argument cannot easily capture the correct mean order of $\Vert BA\Vert$ in the case of  subexponential entries. The important tool we use is a comparison theorem  due to Latała, van Handel, and Youssef \cite{inventions}; see Lemma \ref{ct} below.  It states that 
\begin{align}
	\textsf{E}\Vert BA\Vert^{p}\le C^{p}\textsf{E}\Vert B(g_{ij}g^{\prime}_{ij})_{N\times n}\Vert^{p},\nonumber
\end{align}
where $\{ g_{ij}, g^{\prime}_{ij}\}$ are independent gaussian variables.

Conditioned on $\{g^{\prime}_{ij}\}$, $(g_{ij}g^{\prime}_{ij})_{N\times n}$ becomes a structured gaussian matrix.  Hence, the $\varepsilon$-net method contributes to a proper bound of $\textsf{E}_{g_{ij}}\Vert BA\Vert^{p}$, which is $\Xi^{ p}_{1}(\sqrt{n}+\sqrt{p})^{p}$ and $\Xi_{1}$ is a random variable containing information about $\{ g^{\prime}_{ij}\}$. In particular,
\begin{align}
	\Xi_{1}^{2}=\sum_{i\le N}(\max_{j\le n}g^{\prime2}_{ij})\Vert B_{i}\Vert^{2}_{2},\nonumber
\end{align}
where $B_{i}$ is the $i$th columns of $B$. 

We only need to give a proper bound of $\Xi_{1}$. Using the comparison theorem again, we can substitute $\max_{j\le n}\eta_{j}$ for $\max_{j\le n}g^{\prime2}_{ij}$ in $\Xi_{1}$, where $\{\eta_{j}\}$ is a sequence of exponential variables with parameter 1.
It is well-known that $\max_{j\le n}\eta_{j}$ is identical in distribution to a linear sum of a sequence of independent chi-square variables, and so Bernstein's inequality is available.  Unfortunately, doing this   shall produce an extra factor $\sqrt{n}$ in controlling $\textsf{E}\Xi_{1}^{p}$. To eliminate the $\sqrt{n}$ term, we need to further appropriately truncate $\Xi_{1}$.  With this procedure, we prove our main result, Theorem \ref{Theo_main_weak}.

 
The rest of this paper is organized as follows: Section 2 recalls some basic concepts and notations and lists some lemmas about the concentration inequalities about subgaussian and subexponential random variables. The so-called comparison theorem is stated as Lemma 2.8. Section 3 is devoted to the proof of our main result.

\section{Preliminaries}
\subsection{Notations}

Unless otherwise stated, we denote by $C, C_{1}, c, c_{1},\cdots$ universal constants which are independent of the dimensions of matrices and the parameters of random variables, and by $C(\varepsilon), C_{1}(\varepsilon),\cdots$ constants that depend only on the parameter $\varepsilon$. Their values can change from line to line.

The standard inner product in $\mathbb{R}^{n}$ is denoted $\left\langle x, y\right\rangle $. For a vector $x\in \mathbb{R}^{n}$, we denote the $l_{p}$ norm by $\Vert x\Vert_{p}=(\sum_{i}x_{i}^{p})^{1/p}$, where $p\ge 1$. The unit ball with respect to $l_{p}$ norm in $\mathbb{R}^{n}$ is denoted by $B_{p}^{n}=\{x: \Vert x\Vert_{p}\le 1\}$, and the corresponding unit sphere in $\mathbb{R}^{n}$ is denoted by $S_{p}^{n-1}=\{x: \Vert x\Vert_{2}= 1\}$.

We write $a\lesssim b$ for convenience if a universal constant $C$ satisfies $a\le Cb$. We write $a\asymp b$ if $a\lesssim b$ and $b\lesssim a$.

\subsection{Concentration of subgaussian variables}

In this subsection, we shall introduce some concentration inequalities of subgaussian variables, particularly, Gaussian and bounded variables, but omit their proofs. The interested reader is  referred to \cite{concentration} for confirmation and more information.

\begin{mylem}[Gaussian concentration]
	Let $f : \mathbb{R}^{m}\to \mathbb{R}$ be a Lipschitz function. Let $g$ be a standard Gaussian random vector in $\mathbb{R}^{m}$. Then for every $t>0$ one has
	\begin{align}
		\textsf{P}\{ f(g)-\textsf{E}f(g)>t\}\le \exp(-c_{0}t^{2}/\Vert f\Vert^{2}_{\text{Lip}}),
	\end{align}
	where $c_{0}\in (0, 1)$ is a universal constant.
\end{mylem}

The subsequent concentration inequality is for bounded random variables, known as Talagrand's concentration inequality.

\begin{mylem}
	Let $X_{1},\cdots, X_{m}$ be independent random variables such that $\vert X_{i}\vert\le K$ for all $i$. Let $f : \mathbb{R}^{m}\to \mathbb{R}$ be a convex and $1$-Lipschitz function. Then for every $t>0$ one has
	\begin{align}
		\textsf{P}\{ \vert f(X_{1}, \cdots, X_{m})-\textsf{E}f(X_{1}, \cdots, X_{m})\vert>Kt\}\le 4\exp(-t^{2}/4).
	\end{align}
\end{mylem}

\subsection{Subexponentional random variables}

\begin{mylem}[Proposition 2.7.1 in \cite{highdimension}]\label{Lem_2.3}
	Let $\eta$ be a subexponential random variable. Then the following properties are equivalent.
	
	(i)  $\textsf{P}\{\vert \eta\vert\ge t \}\le 2\exp(-t/K_{1})$ for all $t\ge 0$;
	
	(ii)	 $(\textsf{E}\vert \eta\vert^{p})^{1/p}\le K_{2}p$ for all $p\ge 1$;
	
	(ii)  $\textsf{E}\exp(\lambda\vert\eta\vert)\le \exp(K_{3}\lambda)$ for all $\lambda$ such that $0\le \lambda\le \frac{1}{K_{3}}$;
	
	(iv) $\textsf{E}\exp(\vert\eta\vert/K_{4})\le2$.

	(v) If $\textsf{E}\eta=0$, we also have the following equivalent property:
	
	$\textsf{E}\exp(\lambda X)\le \exp(K_{5}^{2}\lambda^{2})$  for all $\lambda$ such that $\vert \lambda\vert\le \frac{1}{K_{5}}$.
	
\end{mylem}
We remark that the parameters $K_{1},\cdots, K_{5}$ that appeared in the above lemma  are not universal constants. In fact, $K_{i}=C_{i}\Vert \eta\Vert_{\psi_{1}}$ for $i=1,\cdots, 5$, where $C_{1},\cdots ,C_{5}$ are universal constants. Moreover, $K_{i}\lesssim K_{j}$ for any $i, j\in \{1, 2, \cdots, 5\}$.

The following lemma, Bernstein's inequality, is a concentration inequality for sums of independent subexponential random variables.

\begin{mylem}[Theorem 2.8.1 in \cite{highdimension}]\label{Lem_Berstein}
	Let $\eta_{1},\cdots,\eta_{n}$ be independent mean zero subexponential random variables, and let $a=(a_{1},\cdots, a_{n}) \in \mathbb{R}^{n}$. Then, for every $t\ge 0$, we have
	\begin{align}
		\textsf{P}\big\{ \sum_{i=1}^{n}a_{i}\eta_{i}\ge t\big\}\le \exp\big(-c\min(\frac{t^{2}}{K^{2}\Vert a\Vert_{2}^{2}}, \frac{t}{K\Vert a\Vert_{\infty}})\big),\nonumber
	\end{align}
	where $K=\max_{i}\Vert \eta_{i}\Vert_{\psi_{1}}$.
\end{mylem}

The next lemma is an interesting observation on the order statistics of $n$ independent exponential distributed random variables.   For the sake of completeness, we include a simple proof.
\begin{mylem}
	Let $\eta_{1},\cdots, \eta_{n}$ be a sequence of independent exponential random variables with parameter $\lambda =1$. Denote their order statistics by $\eta_{(1)}\le \cdots \le \eta_{(n)}$. Consider the following system of linear transformations,
	\begin{align}
		T_{1}=2n\eta_{(1)},\,  T_{2}=2(n-1)(\eta_{(2)}-\eta_{(1)}),\,\cdots, \, T_{n}=2(\eta_{(n)}-\eta_{(n-1)}).\nonumber
	\end{align}
	Then $\{T_{i}\}$ are   independent chi-square  random variables with degrees of freedom  2.
\end{mylem}

\begin{myrem}
	It is easy to see
	\begin{align}
		\eta_{(n)}=\sum_{i=1}^{n}\frac{T_{i}}{2(n-i+1)},\nonumber
	\end{align}
	which in turn  implies
	\begin{align}
		\textsf{E}\eta_{(n)}=\textsf{E}\sum_{i=1}^{n}\frac{T_{i}}{2(n-i+1)}=\sum_{i=1}^{n}\frac{1}{n-i+1}=\sum_{i=1}^{n}\frac{1}{i}\asymp\log n.\nonumber
	\end{align}
\end{myrem}
\begin{proof}
	Denote by $f(x_{1}, \cdots, x_{n})$ the joint density function of $(\eta_{(1)}, \cdots, \eta_{(n)})$. Then, we have
	\begin{align}
		f(x_{1}, \cdots, x_{n})=n!\cdot\exp\big(-\sum_{i\le n}x_{i}\big)\mathbf{1}_{\{ x_{1}\le \cdots \le x_{n} \}}.\nonumber
	\end{align}
	Note that
	\begin{align}
		\exp(-\sum_{i\le n}x_{i})=\exp\big((x_{n}-x_{n-1})+2(x_{n-1}-x_{n-2})+ \cdots +nx_{1}\big).\nonumber
	\end{align}
	
\end{proof}

\subsection{Moments and Tails}

In this subsection, we introduce some connections between tails and moments and present a fundamental lemma, i.e., a comparison theorem proposed by Latała, van Handel, and Youssef \cite{inventions}.  It may be regarded as an analog of the classic contraction principle; see Lemma 4.6 in Ledoux and Talagrand \cite{proinbanach}.

\begin{mylem}[Lemma 4.6 in \cite{inventions}]
	Let $h$ be a random variable such that
	\begin{align}\label{jutiaojian}
		K_{1}p^{\beta}\le (\textsf{E}\vert h\vert^{p})^{1/p}\le K_{2}p^{\beta}, \quad\text{for all}\,\, p\ge 2.
	\end{align}
	Then there exist constants $K_{3}, K_{4}$ depending only on $K_{1}, K_{2}, \beta$ such that
	\begin{align}
		K_{3}e^{-t^{1/\beta}/K_{3}}\le \textsf{P}\{\vert h\vert\ge t\}\le K_{4}e^{-t^{1/\beta}/K_{4}},\quad\text{for all}\,\, p\ge 2.\nonumber
	\end{align}
\end{mylem}
Indeed, if $h$ satisfies for $p\ge 2$
\begin{align}
	(\textsf{E}\vert h\vert^{p})^{1/p}\le K_{2}p^{\beta},\nonumber
\end{align}
then Markov's inequality  immediately yields the upper tail
\begin{align}
	\textsf{P}\{\vert h\vert\ge t\}\le K_{4}e^{-t^{1/\beta}/K_{4}}.\nonumber
\end{align}
In order to bound the lower tail, one can invoke the Paley-Zygmund inequality, which in turn requires that $h$ satisfies the following two-side inequalities
\begin{align}
	K_{1}p^{\beta}\le (\textsf{E}\vert h\vert^{p})^{1/p}\le K_{2}p^{\beta}, \quad\text{for all}\,\, p\ge 2.\nonumber
\end{align}

\begin{mylem}[Lemma A.1 in \cite{ejp}]\label{Lem_need}
	If $\eta$ is a random variable satisfying for $p\ge 1$
	\begin{align}
		(\textsf{E}\vert \eta\vert^{p})^{1/p}\le a_{1}p+a_{2}\sqrt{p}+a_{3},\nonumber
	\end{align}
	where $0\le a_{1}, a_{2}, a_{3}<\infty$. Then, we have for $u\ge 1$
	\begin{align}
		\textsf{P}\big\{\vert \eta\vert\ge C(a_{1}u+a_{2}\sqrt{u}+a_{3}) \big\}\le e^{-u}.\nonumber
	\end{align}
\end{mylem}

The following result is   an extension of Lemma 4.7 in \cite{inventions}, which establishes the comparison principle for arbitrary random variables with same-level central moments. 
The difference between Lemma \ref{ct} and Lemma 4.7  in \cite{inventions} is that, we only control the moments of   $h_i$  from the above here.

\begin{mylem}[Comparison Theorem]\label{ct}
	Let $h_{i}$ and $h_{i}'$, $i=1, \cdots, n$ be independent centered random variables such that
	\begin{align}
		(\textsf{E}\vert h_{i}\vert^{p})^{1/p}\le K_{0}p^{\beta}, \qquad  K_{1}p^{\beta}	\le(\textsf{E}\vert h_{i}'\vert^{p})^{1/p}\le K_{2}p^{\beta}\nonumber
	\end{align}
	for all $p\ge 2$ and $i=1,\cdots, n$. Then there exists a constant $K$ depending only on $K_{0}, K_{1}, K_{2}$, and $ \beta$ such that
	\begin{align}
		\textsf{E}f(h_{1},\cdots,h_{n})\le \textsf{E}f(Kh_{1}',\cdots,Kh_{n}' )\nonumber
	\end{align}
	for every symmetric convex function $f: \mathbb{R}^{n}\to \mathbb{R}$.
\end{mylem}

We remark that the {\it symmetric} means that  $f(x)=f(-x)$ for every $x\in\mathbb{R}^{n}$.
Lemma 2.8 is also true for independent positive random variables. Let $\{\varepsilon_{i}, \varepsilon_{i}'\}$ be a sequence of independent Rademacher variables, and $\{h_{i}, h_{i}'\}$ are independent positive variables satisfying the moment assumptions of Lemma 2.8. For every symmetric convex function $f$, let $g(x_{1}, \cdots, x_{n}):=f(\vert x_{1}\vert,\cdots,\vert x_{n}\vert)$. Then $g$ is a symmetric convex function. Hence, we have
\begin{eqnarray}
	\textsf{E}f( h_{1},\cdots, h_{n})&=&\textsf{E}g(\varepsilon_{1}h_{1}, \cdots, \varepsilon_{n}h_{n}) \nonumber\\
	&\le& \textsf{E}g(K\varepsilon_{1}h_{1}', \cdots, K\varepsilon_{n}h_{n}')=\textsf{E}f( \vert K\vert h_{1}',\cdots, \vert K\vert h_{n}')\nonumber.
\end{eqnarray}

\begin{proof}[Proof of Lemma 2.8]
	We begin by noting that as $h_{i}$ are centered, Jensen's inequality yields
	\begin{align}
		\textsf{E}f( h_{1},\cdots, h_{n})&\le \textsf{E}f( h_{1}-\tilde{h}_{1},\cdots, h_{n}-\tilde{h}_{n})\nonumber\\
		&\le  \textsf{E}f\big( \frac{1}{2}(2h_{1})+\frac{1}{2}(-2\tilde{h}_{1}),\cdots, \frac{1}{2}(2h_{n})+\frac{1}{2}(-2\tilde{h}_{n})\big)\nonumber\\
		&\le\textsf{E}f( 2h_{1},\cdots, 2h_{n})
	\end{align}
	whenever $f$ is symmetric and convex, where $\tilde{h}_{i}$ are independent copies of $h_{i}$. Moreover, the random variables $h_{i}-\tilde{h}_{i}$ clearly satisfy the same moment assumptions as $h_{i}$ modulo a universal constant. We can therefore assume without loss of generality in the sequel that the random variables $h_{i}$ and $h_{i}^{\prime}$ are symmetrically distributed.

	To proceed,   it follows from Lemma 2.6   that
	\begin{align}
		\textsf{P}\{ \vert h_{i}\vert\ge t \}\le c\textsf{P}\{ c\vert h_{i}^{\prime}\vert\ge t \}\nonumber
	\end{align}
	for all $t\ge 0$ and $i$, where $c\ge 1$ is a constant depends only on $K_{0}, K_{1}, K_{2}$, and $\beta$. Let $\delta_{i}\sim \text{Bern}(1/c)$ be i.i.d. Bernoulli variables independent of $h$. Then
	\begin{align}
		\textsf{P}\{ \delta_{i}\vert h_{i}\vert\ge t \}\le \textsf{P}\{ c\vert h_{i}^{\prime}\vert\ge t \}\nonumber
	\end{align}
	for all $t\ge 0$. By a standard coupling argument, we can couple $(\delta_{i}, h_{i})$ and $(h_{i}^{\prime})$ on the same probability space such that $\delta_{i}\vert h_{i}\vert\le c\vert h_{i}^{\prime}\vert$ a.s. for every $i$. Interested readers can refer to Page 127 of \cite{lectures} for this coupling argument.

	Let $\varepsilon_{i}$ be i.i.d. Rademacher variables. Note that the function $$x\mapsto \textsf{E}f(x_{1}\varepsilon_{1},\cdots,  x_{n}\varepsilon_{n})$$ is convex, so its supremum over $\prod_{i\le n}[-c\vert h_{i}^{\prime}\vert, c\vert h_{i}^{\prime}\vert ]$ is attained at one of the extreme points. We also note that
	\begin{align}
		\textsf{E}f(x_{1}\varepsilon_{1}, x_{2}\varepsilon_{2}, \cdots,  x_{n}\varepsilon_{n})=\textsf{E}f(-x_{1}\varepsilon_{1}, x_{2}\varepsilon_{2}, \cdots,  x_{n}\varepsilon_{n})\nonumber
	\end{align}
	due to that $\varepsilon_{i}$ are symmetric. Hence, we have
	\begin{align}
		\textsf{E}\big(f(\varepsilon_{1}\delta_{1}\vert h_{1}\vert, \cdots, \varepsilon_{n}\delta_{n}\vert h_{n}\vert)\big| \delta, h, h^{\prime}\big)\le \textsf{E}\big(f(c\varepsilon_{1}\vert h_{1}^{\prime}\vert, \cdots, c\varepsilon_{n}\vert h_{n}^{\prime}\vert)\big| \delta, h, h^{\prime}\big)\nonumber,
	\end{align}
	which implies that
	\begin{align}
		\textsf{E}f(\delta_{1}h_{1},\cdots, \delta_{n}h_{n})\le \textsf{E}f(ch_{1}^{\prime},\cdots, ch_{n}^{\prime}).\nonumber
	\end{align}
	Using Jensen's inequality again, we have
	\begin{align}
		\textsf{E}f(\delta_{1}h_{1},\cdots, \delta_{n}h_{n})=	\textsf{E}\textsf{E}_{\delta}f(\delta_{1}h_{1},\cdots, \delta_{n}h_{n})\ge \textsf{E}f(h_{1}/c,\cdots, h_{n}/c)\nonumber.
	\end{align}
	This concludes the proof.
\end{proof}

\begin{myrem}\label{Rem_contraction}
	From the proof one can see that, if $\{h_{i}, h^{\prime}_{i} \}$ satisfies for $t\ge 0$ and $c\ge 1$
	\begin{align}
		\textsf{P}\{ \vert h_{i}\vert\ge t \}\le c\textsf{P}\{ c\vert h_{i}^{\prime}\vert\ge t \},\nonumber
	\end{align}
we have 
\begin{align}
	\textsf{E}f(h_{1},\cdots,h_{n})\le \textsf{E}f(Kh_{1}',\cdots,Kh_{n}' )\nonumber
\end{align}
for every symmetric convex function $f: \mathbb{R}^{n}\to \mathbb{R}$.

\end{myrem}

\subsection{Nets}
In this  subsection, we introduce some properties about $\varepsilon$-nets. Interested readers can refer to Section 4 in \cite{highdimension} for more information.

Let $X$ be a normed space and $U\subset X$. We call a subset $\mathcal{N}\subset U$ an $\varepsilon$-net of $U$ if satisfying
\begin{align}
	\text{dist}(u, \mathcal{N})<\varepsilon, \qquad \forall u\in U.\nonumber
\end{align}

\begin{mylem}
	Let $\varepsilon\in (0, 1)$. The unit Euclidean ball $B_{2}^{n}$ and the unit Euclidean sphere $S^{n-1}$ in $\mathbb{R}^{n}$ both have $\varepsilon$-nets of cardinality at most $2n(1+2/\varepsilon)^{n-1}$.
\end{mylem}

The follwoing lemma shows that $\varepsilon$-nets play an important role in computing operator norms.
\begin{mylem}
	Let $A: X\to Y$ be a linear operator between normed spaces $X$ and $Y$, and let $\mathcal{N}$ be an $\varepsilon$-net of either the unit sphere $S(X)$ or the unit ball $B(X)$ of $X$ for some $\varepsilon\in (0, 1)$. Then
	\begin{align}
		\Vert A\Vert\le \frac{1}{1-\varepsilon}\sup_{x\in \mathcal{N}}\Vert Ax\Vert_{Y}.\nonumber
	\end{align}
\end{mylem}

\section{Proof of the main result}
\subsection{Proof of Theorem \ref{Theo_main_weak}}
\subsubsection{Tail bounds for matrices with small columns}
In this subsection, we shall show a tail bound with exponential decay but under the additional assumption that the norms of  columns of the matrix $B=(B_{1}, \cdots, B_{N})$ are  small.

\begin{mytheo}
	Let $N, n$ be positive integers. Consider an $N\times  n$ random matrix $A$ whose entries are subexponentional random variables with mean zero and $\max_{i,j}\Vert a_{ij}\Vert_{\psi_{1}}\le 1$. Let $B$ be an $n\times N $ matrix such that $\Vert B\Vert\le 1$, and whose columns satisfy for every $i$
	\begin{align}
		\Vert B_{i}\Vert_{2}\lesssim n^{-1}\log^{-\frac{5}{2}}n.\nonumber
	\end{align}
Then we have for $p\ge 1$
\begin{align}
	(\textsf{E} \Vert BA\Vert^{p})^{1/p}\le C\sqrt{n}p.
\end{align}
\end{mytheo}

 To prove Theorem 3.1, we first introduce the following lemma, which is Proposition 3.7 in \cite{ptrfvershynin}

\begin{mylem}
	Let $a, b\ge 0$ and $N, n$ be positive integers. Let $A$ be an $N\times n$ matrix whose entries are random independent variables $a_{ij}$ with mean zero and such that
	\begin{align}
		\textsf{E}a_{ij}^{2}\le 1,\quad \vert a_{ij}\vert\le a.\nonumber
	\end{align}
Let $B$ be an $n\times N$ matrix such that $\Vert B\Vert\le 1$, and whose columns satisfy for every $i\ge 1$
\begin{align}
	\Vert B_{i}\Vert_{2}\le b.\nonumber
\end{align}
Then
\begin{align}
	\textsf{P}\big\{ \Vert BA\Vert>C\big(1+ab^{1/2}\log^{1/4}(2n)\big)\sqrt{n}+at\big\}\le 4\exp(-t^{2}/4)\nonumber.
\end{align}
\end{mylem}

\begin{proof}[Proof of Theorem 3.1]
Without loss of generality, one can assume all entries $a_{ij}$ are symmetric. Let $a=C\sqrt{n}\log n$. Here, $C$ is a suitably large universal constant. Truncate every entry of the matrix $A$ according to $a$ and set
	\begin{align}
		\overline{a}_{ij}:=a_{ij}\mathbb{I}_{(\vert a_{ij}\vert\le a)},\quad \tilde{a}_{ij}:=a_{ij}\mathbb{I}_{(\vert a_{ij}\vert> a)}.\nonumber
	\end{align}
Then all random variables $\overline{a}_{ij}$ and $\tilde{a}_{ij}$ have mean zero. Decompose $BA$ in the following way
\begin{align}
	BA=B\overline{A}+B\tilde{A},\nonumber
\end{align}
where $\overline{A}=(\overline{a}_{ij})$ and $\tilde{A}=(\tilde{a}_{ij})$.

Note that the condition $\max_{i,j}\Vert a_{ij}\Vert_{\psi_{1}}\le 1$ and the subexponential properties (Lemma 2.3) imply $\overline{a}_{ij}$ satisfy the assumptions of Lemma 3.1 with $a=C\sqrt{n}\log n, b=C_{1}n^{-1}\log^{-5/2}n$. Hence, we have for $t\ge 0$
\begin{align}
	\textsf{P}\big\{ \Vert B\overline{A}\Vert>C_{2}\sqrt{n}+C\sqrt{n}t\big\}\le 4\exp(-t^{2}/4).\nonumber
\end{align}
Integration yields for $p\ge 1$
\begin{align}\label{bound}
	(\textsf{E}\Vert B\overline{A}\Vert^{p})^{1/p}\lesssim \sqrt{np}.
\end{align}

Next, we concentrate on the tail bound for $\Vert B\tilde{A}\Vert$. Let $\{g_{ij}, g^{\prime}_{ij}\}$ be a sequence of independent standard normal random variables, and let $$A^{\prime}=(a_{ij}^{\prime}):=(g_{ij}g_{ij}^{\prime}\mathbb{I}_{(\vert g_{ij}^{\prime}\vert>\sqrt{a})}).$$

On the one hand, the function $f: \mathbb{R}^{Nn}\to \mathbb{R}$ defined by $f(A)=\Vert BA\Vert$ is convex and $f(A)=f(-A)$ due to the property of the spectral norm.

On the other hand, we have by conditional probability
\begin{align}
	\big(\textsf{E}\vert g_{ij}g_{ij}^{\prime}\vert^{p}\big)^{1/p}=\big(\textsf{E}(\textsf{E}_{(g_{ij})}\vert g_{ij}g_{ij}^{\prime}\vert^{p})\big)^{1/p}\asymp p,\nonumber
\end{align}
 which means by Lemma 2.6
\begin{align}
\textsf{P}\{\vert a_{ij}\vert >t  \}\le c\textsf{P}\{c\vert g_{ij}g^{\prime}_{ij}\vert >t  \},\nonumber
\end{align}
where $t\ge 0$ and $c\ge 1$.
Recalling the definition of $\tilde{a}_{ij}$ and $a^{\prime}_{ij}$, we have for $t\ge 0$
\begin{align}
	\textsf{P}\{\vert \tilde{a}_{ij}\vert>t  \}\le c\textsf{P}\{c\vert a^{\prime}_{ij}\vert>t  \}, \nonumber
\end{align}
due to $\mathbb{I}_{(\vert g_{ij}g_{ij}^{\prime}\vert >a)}\le \mathbb{I}_{(\vert g_{ij}\vert >\sqrt{a})}+\mathbb{I}_{(g_{ij}^{\prime}\vert >\sqrt{a})} $.
Then, Remark \ref{Rem_contraction} yields that
\begin{align}
	(\textsf{E}\Vert B\tilde{A}\Vert^{p})^{1/p}\lesssim (\textsf{E}\Vert BA^{\prime}\Vert^{p})^{1/p}, \quad p\ge 1.\nonumber
\end{align}
Let $W_{1}^{\prime}, \cdots, W_{n}^{\prime}\in \mathbb{R}^{n}$ denote the columns of the matrix $BA^{\prime}$, then for $j=1,\cdots,n$
\begin{align}
	W_{j}^{\prime}=\sum_{i\le N}g_{ij}\tilde{g}_{ij}B_{i},\nonumber
\end{align}
where $\tilde{g}_{ij}=g_{ij}^{\prime}\mathbb{I}_{(\vert g_{ij}^{\prime}\vert>\sqrt{a})}$.
Consider a $(1/2)$-net $\mathcal{N}$ of the unit Euclidean sphere $\mathbb{S}^{n-1}$ of cardinality $\vert \mathcal{N}\vert\le 5^{n}$, then we have by Lemma 2.10
\begin{align}
	\Vert BA^{\prime}\Vert^{2}=\Vert (BA^{\prime})^{*}\Vert^{2}\le 4\max_{x\in \mathcal{N}}\Vert(BA^{\prime})^{*}x\Vert_{2}^{2}=4\max_{x\in \mathcal{N}}\sum_{j\le n}\langle W_{j}^{\prime}, x\rangle^{2}. \nonumber
\end{align}

For every $j\le n$, conditioning on $\{ \tilde{g}_{ij}\}$, the random variable
\begin{align}
	\langle W_{j}^{\prime}, x\rangle=\sum_{i\le N}g_{ij}\left\langle \tilde{g}_{ij}B_{i}, x\right\rangle\nonumber
\end{align}
is a Gaussian random variable with mean zero and variance given by
\begin{align}
	\sum_{i\le N}\left\langle \tilde{g}_{ij}B_{i}, x\right\rangle^{2}\le \sum_{i\le N}(\max_{j\le n}\tilde{g}_{ij}^{2})\Vert B_{i}\Vert_{2}^{2}=:\Xi^{2}.\nonumber
\end{align}

To obtain the concentration inequality for $\sum_{j\le n}\langle W_{j}^{\prime}, x\rangle^{2}\big)^{1/2}$, we employ  Lemma 2.1 for  the function $g(y)=(\sum_{j\le n}d_{j}^{2}y_{j}^{2})^{1/2}$. Note that $g$ is a Lipschitz function on $\mathbb{R}^{n}$ with $\Vert g\Vert_{\text{Lip}}=\Vert d\Vert_{\infty}$.
Therefore, by Lemma 2.1 with $d_{i}\le \Xi$, we have for every $t\ge 0$
\begin{align}
	\textsf{P}_{g}\Big\{\big(\sum_{j\le n}\langle W_{j}^{\prime}, x\rangle^{2}\big)^{1/2}>\Xi\sqrt{n}+t\Big\}\le e^{-c_{0}t^{2}/\Xi},\nonumber
\end{align}
where $\textsf{P}_{g}$ is with respect to $(g_{ij})$ (i.e. conditioned on $(\tilde{g}_{ij})$). If we let $t=s\Xi\sqrt{n}$, then we have
\begin{align}
	\textsf{P}_{g}\Big\{\big(\sum_{j\le n}\langle W_{j}^{\prime}, x\rangle^{2}\big)^{1/2}>\Xi\sqrt{n}(1+s)\Big\}\le e^{-c_{0}s^{2}n},\nonumber
\end{align}
Taking the union bound over $x\in \mathcal{N}$, we obtain
\begin{align}
	\textsf{P}_{g}\{  \Vert BA^{\prime}\Vert >2(1+s)\Xi\sqrt{n}\}\le 5^{n}e^{-c_{0}s^{2}n}\le e^{(2-c_{0}s^{2})n}.\nonumber
\end{align}
By adjusting the constant, we have for $u\ge 1$
\begin{align}
	\textsf{P}_{g}\{  \Vert BA^{\prime}\Vert >Cu\Xi\sqrt{n}\}\le e^{-cu^{2}n}.\nonumber
\end{align}
Integration yields
\begin{align}\label{conditional}
	\textsf{E}_{g}\Vert BA^{\prime}\Vert^{p}\le C^{p}\Xi^{p}(\sqrt{n}+\sqrt{p})^{p},
\end{align}
where $p\ge 1$ and $C$ is a universal constant.

Then we turn to the random variable $\Xi$. For $t\ge 0$ and $\lambda>0$, we have by the Cram\'{e}r-Chernoff bounding method
\begin{align}\label{3.2}
	\textsf{P}\big\{ \Xi^{2}-\textsf{E}\Xi^{2}>t \big\}\le e^{-\lambda t}\prod_{i=1}^{N}\textsf{E}\exp\Big(\lambda \big(\max_{j\le n}\tilde{g}_{ij}^{2}-\textsf{E} \max_{j\le n}\tilde{g}_{ij}^{2}\big)\Vert B_{i}\Vert_{2}^{2}\Big).
\end{align}
Let $\eta_{1}, \cdots,\eta_{n}$ be a sequence of exponential random variables with parameter $1$. Note that by Remark 2.1,
\begin{align}
	 \Big(\textsf{E}\big\vert \max_{j\le n}\eta_{j}-\textsf{E}\max_{j\le n}\eta_{j}  \big\vert^{p}\Big)^{1/p}\asymp p .\nonumber
\end{align}
By Markov's inequality, we have for $p\ge 1$
\begin{align}
	\textsf{E}\vert\tilde{g}_{ij}\vert^{p}\le a^{-p/2}\textsf{E}\vert g^{\prime}_{ij}\vert^{p}\lesssim (\frac{p}{a})^{p/2}.\nonumber
\end{align}
Hence, we have
\begin{align}
a \Big(\textsf{E}\big\vert \max_{j\le n}\tilde{g}_{ij}^{2}-\textsf{E}\max_{j\le n}\tilde{g}_{ij}^{2}  \big\vert^{p}\Big)^{1/p}\lesssim 	a\Big(\textsf{E} (\max_{j\le n}\tilde{g}_{ij}^{2})^{p}\Big)^{1/p}\lesssim \Big(\textsf{E} (\max_{j\le n}\eta_{j})^{p}\Big)^{1/p}\lesssim p\log n.\nonumber
\end{align}
We have by Lemma 2.5
\begin{align}
	\Big(\textsf{E} (\max_{j\le n}\eta_{j})^{p}\Big)^{1/p}=\Big(\textsf{E} (\sum_{l=1}^{n}\frac{T_{l}}{2l})^{p}\Big)^{1/p}\lesssim p\log n,\nonumber
\end{align}
where $\{T_{l}\}$ are independent chi-square distributed random variables whose degrees of freedom are  $2$.  Hence, we have
\begin{align}
	\Big(\textsf{E}\big\vert \max_{j\le n}\tilde{g}_{ij}^{2}-\textsf{E}\max_{j\le n}\tilde{g}_{ij}^{2}  \big\vert^{p}\Big)^{1/p}\lesssim \frac{p}{\sqrt{n}}.\nonumber
\end{align}

 For $\lambda\le \sqrt{n}/(C_{1}\max_{i}\Vert B_{i}\Vert_{2}^{2})$, the property (v) of Lemma 2.3 implies (\ref{3.2}) can be further  bounded by

\begin{align}\label{3.4}
	&e^{-\lambda t}\prod_{i=1}^{N}\textsf{E}\exp\big(C_{0}\lambda \Vert B_{i}\Vert_{2}^{2}(\max_{j\le n}\tilde{g}_{ij}^{2}-\textsf{E}\max_{j\le n}\tilde{g}_{ij}^{2} )\big)\nonumber\\
	\le& \exp\big(-\lambda t+\frac{C_{1}^2\lambda^2}{n}\sum_{i\le N}\Vert B_{i}\Vert_{2}^{4}\big).
\end{align}
To optimize the bound of (\ref{3.4}), let
\begin{align}
	\lambda=\min\Big(\frac{nt}{2C_{1}^{2}\sum_{i\le N}\Vert B_{i}\Vert_{2}^{4}},  \frac{\sqrt{n}}{C_{1}\max_{i}\Vert B_{i}\Vert_{2}^{2}}\Big).\nonumber
\end{align}
Then (\ref{3.4}) is further bounded by
\begin{align}
	\exp\Big(-C_{3}\min(\frac{\sqrt{n}t}{\max_{i\le N}\Vert B_{i}\Vert_{2}^{2}}, \frac{nt^{2}}{\sum_{i\le N}\Vert B_{i}\Vert_{2}^{4}})\Big).\nonumber
\end{align}
Recalling \eqref{3.2}, we have for $t\ge 0$
\begin{align}
	\textsf{P}\big\{ \Xi^{2}>\frac{t}{\sqrt{n}}+\textsf{E}\Xi^{2} \big\}
\le \exp\Big(-C_{3}\min(\frac{t}{\max_{i\le N}\Vert B_{i}\Vert_{2}^{2}}, \frac{t^{2}}{\sum_{i\le N}\Vert B_{i}\Vert_{2}^{4}})\Big).\nonumber
\end{align}
Following the same line above, we can obtain the same bound for $$	\textsf{P}\big\{ -(\Xi^{2}-\textsf{E}\Xi^{2})>\frac{t}{\sqrt{n}} \big\}.$$
Hence, we have for $t\ge 0$
\begin{align}
\textsf{P}\big\{\big\vert \Xi^{2}-\textsf{E}\Xi^{2}\big\vert >\frac{t}{\sqrt{n}}\big\}
\le 2\exp\Big(-C_{3}\min(\frac{t}{\max_{i\le N}\Vert B_{i}\Vert_{2}^{2}}, \frac{t^{2}}{\sum_{i\le N}\Vert B_{i}\Vert_{2}^{4}})\Big).\nonumber
\end{align}
Thus it follows for $p\ge 1$
\begin{align}\label{3.5}
	\textsf{E}\big\vert \Xi^{2}-\textsf{E}\Xi^{2}\big\vert^{p}\lesssim p^{p}.
\end{align}
Observe that, 
\begin{align}
	\textsf{E}\Xi^{2}&=\sum_{i\le N}\Vert B_{i}\Vert_{2}^{2}\textsf{E}\max_{j\le n}\tilde{g}_{ij}^{2}\le\sum_{i\le N}\Vert B_{i}\Vert_{2}^{2}\sum_{j\le n}\textsf{E}\tilde{g}_{ij}^{2} \nonumber\\
 &\le 2n(\sum_{i\le N}\Vert B_{i}\Vert_{2}^{2})\int_{a}^{\infty}x^{2}\exp(-\frac{x^{2}}{2})\, dx.\nonumber
\end{align}
Note that
\begin{align}
    \sum_{i\le N}\Vert B_{i}\Vert_{2}^{2}=\text{Trace}(B^\top B)\le n\Vert B\Vert^{2}\le n.\nonumber
\end{align}
Hence, we have $$\textsf{E}\Xi^{2}\le 2n^{2}\int_{a}^{\infty}x^{2}\exp(-\frac{x^{2}}{2})\, dx=o(1),$$ which implies with (\ref{conditional}) that
 \begin{align}\label{3.9}
 	(\textsf{E}\Vert BA^{\prime}\Vert^{p})^{1/p}=(\textsf{E}_{\tilde{g}}\textsf{E}_{g}\Vert BA^{\prime}\Vert^{p})^{1/p}\lesssim \sqrt{n}p.
 \end{align}
Combining (\ref{bound}) and (\ref{3.9}), we have
\begin{align}
(	\textsf{E} \Vert BA\Vert^{p})^{1/p} \lesssim\sqrt{n}p\nonumber.
\end{align}
\end{proof}

\subsubsection{Almost square matrices}
In this section, we consider the case of almost square matrices.
\begin{mytheo}
	Let $\varepsilon\in (0, 1)$ and let $N, n$ be positive integers satisfying $n\le N\le n^{3+\varepsilon/10}$. Let $A$ be an $N\times n$ random matrix whose entries are independent subexponential random variables with mean zero. Let $B$ be an $n\times N$ amtrix such that $\Vert B\Vert\le 1$. Then, we have for $p\ge 1$
	\begin{align}
		(\textsf{E}\Vert BA\Vert^{p})^{1/p}\le C(\varepsilon)\sqrt{np}.\nonumber
	\end{align}
\end{mytheo}

Before proving the above result, we first introduce  the following properties.

\begin{mylem}[Theorem 4.1 in \cite{ptrfvershynin}] \label{lm32}
	Let $\varepsilon \in (0, 1), M\ge 1$ and let $N\ge n$ be positive integers such that $\log(2N)\le Mn$. Consider an $N\times n$ random matrix $A$ whose entries are independent random variable $a_{ij}$ with mean zero such that
  \begin{align}
  	\textsf{E}\vert a_{ij}\vert^{2+\varepsilon}\le 1,\qquad \vert a_{ij}\vert\le (\frac{Mn}{\log(2N)})^{\frac{1}{2+\varepsilon}}.\nonumber
  \end{align}
Let $B$ be an $n\times N$ matrix such that $\Vert B\Vert\le 1$. Then, one has for every $t>0$
\begin{align}
	\textsf{P}\{\Vert BA\Vert>(C(\varepsilon)+t)\sqrt{Mn} \}\le 4e^{-t^{2}/4}.\nonumber
\end{align}
In particular, one has for every $p\ge 1$
\begin{align}
	(\textsf{E}\Vert BA\Vert^{p})^{1/p}\le C_{0}(\varepsilon)\sqrt{pMn}.\nonumber
\end{align}
\end{mylem}

\begin{mylem}\label{Lemma3.3}
	Let $\xi$ be a random variable and $K$ be a real number. Then
	\begin{align}
		\textsf{E}(\xi\vert \xi\le K)\le \textsf{E}\xi.\nonumber
	\end{align}
\end{mylem}
\begin{proof}
	Note that
	\begin{align}
		\textsf{E}\xi=\textsf{E}(\xi\vert \xi\le K)\textsf{P}\{\xi\le K \}+\textsf{E}(\xi\vert \xi> K)\textsf{P}\{\xi> K \}.\nonumber
	\end{align}
For simplicity, let $a=\textsf{E}(\xi\vert \xi\le K)$ and $b=\textsf{E}(\xi\vert \xi> K)$. Due to that $\textsf{E}\xi$ is a convex combination of $a$ and $b$ and $a\le K\le b$, we must have $a\le \textsf{E}\xi\le b$.
\end{proof}

\begin{proof}[Proof of Theorem 3.2]
	We may and do assume $\max_{i,j}\Vert a_{ij}\Vert_{\psi_{1}}\le 1$. Define the random variable $M$ by the following equation
	\begin{align}
		\max_{i,j}\vert a_{ij}\vert=(\frac{Mn}{\log(2N)})^{\frac{1}{2+\varepsilon/4}}.\nonumber
	\end{align}
Note that $\{ M\le t\}$ is an intersection of a family of independent events. Therefore, conditioning on this event preserves the independence of the entries of $A$. By Lemma 2.3 and the assumption $\max_{i,j}\Vert a_{ij}\Vert_{\psi_{1}}\le 1$, we do assume
\begin{align}
	\textsf{E}\vert a_{ij}\vert^{2+\varepsilon/4}\le 1.\nonumber
\end{align}
Note that
\begin{align}
	\textsf{E}\big(\vert a_{ij}\vert^{2+\varepsilon/4}\big|M\le t\big)&=\textsf{E}\big(\vert a_{ij}\vert^{2+\varepsilon/4}\big|\vert a_{ij}\vert^{2+\varepsilon/4}\le (\frac{tn}{\log(2N)}),\,\,\,\,\forall i,j\big)\nonumber\\
	&=\textsf{E}\big(\vert a_{ij}\vert^{2+\varepsilon/4}\big|\vert a_{ij}\vert^{2+\varepsilon/4}\le (\frac{tn}{\log(2N)})\big)\nonumber\\
	&\le \textsf{E}\vert a_{ij}\vert^{2+\varepsilon/4}\le 1.\nonumber
\end{align}
By virtue of Lemma 3.2, we have for $t, p\ge 1$
\begin{align}\label{4.1}
	\big(\textsf{E}(\Vert BA\Vert^{p}| M\le t)\big)^{1/p}\le C_{0}(\varepsilon)\sqrt{ptn}.
\end{align}

Next, turn to the probability properties of $M$. By the subexponential property, we have
\begin{align}
	\textsf{P}\big\{ \vert a_{ij}\vert>(t^{2}nN)^{\frac{1}{8+2\varepsilon}}\big\}\le 2\exp\big(-(t^{2}nN)^{\frac{1}{8+2\varepsilon}}\big).\nonumber
\end{align}
Taking the union bound over all $nN$ random variables $a_{ij}$, we obtain
\begin{align}
	\textsf{P}\big\{ \max_{i,j}\vert a_{ij}\vert>(t^{2}nN)^{\frac{1}{8+2\varepsilon}}\big\}\le 2nN\exp\big(-(t^{2}nN)^{\frac{1}{8+2\varepsilon}}\big).
\end{align}
The assumption $N\le n^{3+\varepsilon/10}$ yields that
$
	nN\le \big(\frac{C(\varepsilon)n}{\log(2N)}\big)^{4+\varepsilon/8}.
$
Therefore, we have for $t\ge 1$
\begin{align}
	(t^{2}nN)^{\frac{1}{8+2\varepsilon}}\le \big(\frac{C(\varepsilon)tn}{\log(2N)}\big)^{\frac{1}{2+\varepsilon/4}}.\nonumber
\end{align}
which in turn implies for $t\ge 1$
\begin{align}
	&\textsf{P}\{ M>C(\varepsilon)t\}\le \textsf{P}\{ \max_{i,j}\vert a_{ij}\vert >\big(\frac{C(\varepsilon)tn}{\log(2N)}\big)^{\frac{1}{2+\varepsilon/4}}\}\nonumber\\
	\le &\textsf{P}\{ \max_{i,j}\vert a_{ij}\vert>(t^{2}nN)^{\frac{1}{8+2\varepsilon}}\}\le 2nN\exp(-(t^{2}nN)^{\frac{1}{8+2\varepsilon}}).\nonumber
\end{align}
By adjusting the universal constants, we have for $t\ge 1$
\begin{align}\label{4.3}
	\textsf{P}\{ M>C_{1}(\varepsilon)t\}\le 2\exp(-ct^{\frac{1}{4+\varepsilon}}).
\end{align}

We have for $p\ge 1$
\begin{align}\label{4.4}
	\textsf{E}\Vert BA\Vert^{p}=\textsf{E}\Vert BA\Vert^{p}\mathbb{I}_{(M\le C_{1}(\varepsilon))}+\sum_{k=1}^{\infty}\textsf{E}\Vert BA\Vert^{p}\mathbb{I}_{(2^{k-1}C_{1}(\varepsilon)<M\le 2^{k}C_{1}(\varepsilon))}.
\end{align}
The result (\ref{4.1}) and Lemma \ref{Lemma3.3} yield that
\begin{align}
	\textsf{E}\Vert BA\Vert^{p}\mathbb{I}_{(M\le C_{1}(\varepsilon))}\le\textsf{E}\Big(\Vert BA\Vert^{p}\Big|M\le C_{1}(\varepsilon)\Big)\le C_{2}(\varepsilon)^{p}(\sqrt{np})^{p}.\nonumber
\end{align}
We have by Cauchy-Schwarz inequality and the result (\ref{4.3})
\begin{eqnarray}
	& &\textsf{E}\Vert BA\Vert^{p} {I}_{(2^{k-1}C_{1}(\varepsilon)<M \le  2^{k}C_{1}(\varepsilon))}\nonumber\\
	&\le &\big(\textsf{E}\Vert BA\Vert^{2p} {I}_{(M\le 2^{k}C_{1}(\varepsilon))}\big)^{1/2}\big(\textsf{P}\big\{ M>2^{k-1}C_{2}(\varepsilon)\big\}\big)^{1/2}\nonumber\\
	&\le &C_{2}(\varepsilon)^{p}(\sqrt{np})^{p}(2^{k})^{p/2}\exp(-c2^{\frac{k}{4+\varepsilon}}).\nonumber
\end{eqnarray}
Hence, the result (\ref{4.4}) is further bounded by
\begin{align}
	C_{2}(\varepsilon)^{p}(\sqrt{np})^{p}+\sum_{k=1}^{\infty}C_{2}(\varepsilon)^{p}(\sqrt{np})^{p}(2^{k})^{p/2}\exp(-c2^{\frac{k}{4+\varepsilon}})\le C_{3}(\varepsilon)^{p}(\sqrt{np})^{p}.\nonumber
\end{align}
\end{proof}

We are now ready to prove  our first main result.

\begin{proof}[Proof of Theorem \ref{Theo_main_weak}]
	Let $B_{1},\cdots, B_{N}$ be the columns of the matrix $B$, and denote by $I$ the subset of $\{1,2,\cdots,N\}$ of large columns, namely
	\begin{align}
		I:=\{i: \Vert B_{i}\Vert_{2}>Cn^{-1}\log^{-\frac{5}{2}}n\}.\nonumber
	\end{align}
Note that
\begin{align}
	\sum_{i\le N}\Vert B_{i}\Vert_{2}^{2}\le n\Vert B\Vert^{2}\le n.\nonumber
\end{align}
Hence, we have
\begin{align}
	N_{0}:=|I|\le C^{-2}n^{3}\log^{5} n\le n^{3+\varepsilon/10},\nonumber
\end{align}
where $\varepsilon\in (0,1)$ and $C$ is a sufficiently large universal constant.

Denote by $A_{I}$ the $N_{0}\times n$ submatrix of $A$ whose rows are in $I$, by $B_{I}$ the $n\times N_{0}$  submatrix of $B$ whose columns are in $I$ (and similarly for $I^{c}$). Now $W$ is rewritten as
\begin{align}
	W=BA=B_{I}A_{I}+B_{I^{c}}A_{I^{c}}.\nonumber
\end{align}
Putting Theorems 3.1 and 3.2 together yields the desired result.
\end{proof}



\textbf{Acknowledgment} Su was partly supported by the National Natural Science Foundation of China  (12271475, 11871425) and fundamental research funds for central universities grants. Wang was partly supported by the Shandong Provincial Natural Science Foundation (No. ZR2024MA082) and the National Natural Science Foundation of China (No. 12071257, 12371148).


\end{document}